\documentclass[12pt, a4wide]{amsart}
\usepackage{ifthen,amssymb,graphicx}
\usepackage[all]{xy}
\usepackage[usenames]{color}
\usepackage{mathrsfs}

\newcommand{\catD}{{D}}
\newcommand{\dbc}[1]{{D}^b(#1)}
\newcommand{\dpc}[1]{{D}^+(#1)}
\newcommand{\dmc}[1]{{D}^-(#1)}
\newcommand{\cdbc}[1]{{D}^b_c(#1)}

\newcommand{\Hom}{{\operatorname{Hom}}}
\newcommand{\im}{{\operatorname{Im}}}

\newcommand{\Perf}{{\operatorname{Perf }}}
\newcommand{\VB}{{\operatorname{VB }}}

\newcommand{\FM}{{\operatorname{FM}}}

\newcommand{\sCM}{{\underline{\text{CM}}}}
\newcommand{\CM}{{\text{CM}}}

\newcommand{\lotimes}{{\,\stackrel{\mathbf L}{\otimes}\,}}

\DeclareMathOperator{\Pic}{{Pic}}

\DeclareMathOperator{\Spec}{{Spec}}

\DeclareMathOperator{\Img}{{Im}}

\newcommand{\bbQ}{{\mathbb Q}}
\newcommand{\bbZ}{{\mathbb Z}}
\newcommand{\bbP}{{\mathbb P}}
\newcommand{\cO}{{\mathcal O}}

\newcommand{\bR}{{\mathbf R}}

\newcommand{\cplx}[1]{{{\mathcal #1}^{\scriptscriptstyle\bullet}}}

\newcommand{\iso}{{\,\stackrel {\textstyle\sim}{\to}\,}}

\newtheorem{thm}{Theorem}[section]
\newtheorem*{thm*}{Theorem}
\newtheorem{cor}[thm]{Corollary}
\newtheorem{lem}[thm]{Lemma}
\newtheorem{prop}[thm]{Proposition}

\theoremstyle{definition}

\theoremstyle{remark}
\newtheorem{rema}[thm]{Remark}
\newtheorem{exe}[thm]{Example}

\numberwithin{equation}{section} %%
\begin{document}
\title{Derived equivalences and Kodaira fibers}

\thanks {{\it Author's address: }Departamento de Matem\'aticas, Universidad de Salamanca, Plaza de la Merced 1-4, 37008, Salamanca, tel: +34 923294456; fax +34 923294583.\\
Work supported by the research project MTM2013-45935-P (MINECO)}
\date{\today}

%\begin{abstract}  
%\end{abstract}

\author[A. C. L\'opez Mart\'{\i}n]{Ana Cristina L\'opez Mart\'{\i}n}
\email{anacris@usal.es}
\author[C. Tejero Prieto]{Carlos Tejero Prieto}
\email{carlost@usal.es}
\address{Departamento de Matem\'aticas and Instituto Universitario de F\'{\i}sica Fundamental y Matem\'aticas
(IUFFyM), Universidad de Salamanca, Plaza de la Merced 1-4, 37008
Salamanca, Spain.}

%\date{\today}
%\thanks {}
\begin{abstract}  
We give necessary conditions for two (including non-reduced and multiple) Kodaira curves to be derived equivalent.
We classify Fourier-Mukai partners of any reduced Kodaira curve. We prove that the derived category of singularities of any non-reduced and non-multiple Kodaira curve is idempotent complete.  \end{abstract}

\maketitle
\setcounter{tocdepth}{1}
%%%%%%%%%%

%{\small \tableofcontents }

\section{Introduction}

For a scheme $X$, we denote by $\text{FM}(X)$ the set of isomorphism classes of (Fourier-Mukai) partners of $X$, that is,
$$\text{FM}(X):=\{ Y \text{ schemes  such that } \catD(X)\simeq \catD(Y)\}/\simeq.$$ 

When one considers  $\catD(X)$  as a tensor triangulated category, a result by Thomason \cite{Thom97} proves that $( \catD(X), \otimes, \mathcal{O}_X)$ is rigidly-compactly generated and then we can recover completely the scheme $X$ from the monoidal structure of $ \catD(X)$. Thus, if $\Phi\colon \catD(X)\simeq \catD(Y)$ is an equivalence respecting $\otimes$-products, then $X\simeq Y$, and consequently $\text{FM}(X)=\{ X\}$ for every scheme $X$.  Of course, this is no longer true when one forgets the tensor structure of $\catD(X)$ and, besides its interest for applications in  Physics, it is an interesting mathematical problem to determine the set $\text{FM}(X)$ for a given $X$. In the last thirty years, much work has been done in this direction  and by now there are many classical results. The first examples of non-isomorphic (Fourier-Mukai) partners were constructed among abelian varieties
and K3 surfaces \cite{Muk81, Muk87a, Or97, Pol96}, varieties connected by some kinds of flops \cite{BO95, Bri02, Chen02}, and many others where $Y$ is a moduli space of certain kind of sheaves on $X$ \cite{BBH08, Bri98, BrM02}. To the contrary, Bondal and Orlov show in \cite{BO01}
that if $X$ is a smooth projective variety with ample or antiample canonical sheaf, then $X$ can be entirely reconstructed
from the $k$-linear graded structure of
 $\cdbc{X}$, so that  in this case  $\text{FM}(X) = \{X\}$.  Other important contributions to this problem are due to Bridgeland-Maciocia \cite{BrM01}, Kawamata \cite{Kaw02a, Kaw02b}, Uehara \cite{Ue04, Ue12, Ue15}, Favero \cite{Fav} and Fabrice \cite{Fabri}. Most of these articles are devoted to smooth projective varieties.
 
In the present article we are concerned with  the problem of determining the Fourier-Mukai partners of  (not necessarily neither smooth nor reduced) one-dimensional \ schemes. Although much less is known for singular varieties, we do have a generalization of  the classical Bondal-Orlov reconstruction theorem for Gorenstein schemes due to  C. and F. Sancho de Salas \cite{SS10}. Like in the smooth context, their result shows that  the most interesting case to study, at least for Gorenstein schemes, is again the case of Fano or anti-Fano schemes.  Thus we focus our attention in the case of projective Gorenstein curves of genus one with trivial dualising sheaf.

Let $k$ be an algebraically closed  field of characteristic zero (see \cite{AnKrWa} for some results about Fourier-Mukai partners of genus one curves over  an non-algebraically closed field). The computation of Fourier-Mukai partners for genus one curves over $k$ started with the result by Hille and Van den Bergh who deal with the classical case of smooth elliptic curves. They prove \cite{HiVdB} that any Fourier-Mukai partner of a smooth elliptic curve over $k$ is isomorphic to itself.
In \cite{LM14}, the first author of this article extends that result to the case of Gorenstein reduced curves.

In this paper, we continue our study of Fourier-Mukai partners for other genus one Gorenstein curves. More concretely, we are interested in answering the following question:  is it possible to have non-isomorphic Kodaira curves with equivalent derived categories? 
In this line we prove Theorem \ref{t1} and Theorem \ref{t2}. Part (1) in Theorem \ref{t1} is new since we consider all Kodaira fibers, which includes curves that are non-reduced  and even multiple curves. Part (2) provides a new proof for the result given in \cite{LM14}. We no longer need to use neither the equivalence $\overline{D_\text{sg}(X)}\simeq \oplus_{i=1}^n\sCM(\hat{\mathcal{O}} _{x_i})$ given in Theorem \ref{t:equivalencia} nor the classification of Cohen-Macaulay modules. As an important technical tool we see how the Picard group is determined from the derived category.   On the other hand, very little is known about the category of singularities $D_\text{sg}(X)$ for a scheme $X$ with non-isolated singularities. If $X$ is a non-reduced and non-multiple Kodaira curve, Theorem \ref{t2} proves that $D_\text{sg}(X)$ is idempotent complete.
Even if our main results are just for curves, all along the paper we provide interesting results concerning the general theory of integral functors and Fourier-Mukai partners. 

Kodaira curves are defined as the possible fibers appearing in a smooth elliptic surface. They are really important because much of what is true for surfaces generalizes to a higher dimensional elliptic fibration, that is, a projective flat morphism  $Z\to B$ of schemes whose generic fiber is a smooth elliptic curve. Roughly speaking,  what is true puntually on the base curve in the case of surfaces becomes true generically in codimension 1. For instance, Kodaira's classification of singular fibers works over the generic point of each irreducible component of the discriminant locus of the fibration. In the case of elliptic threefolds  Miranda even proved that the non-Kodaira fibers are contractions of Kodaira fibers. On the other hand, one should point out that elliptic fibrations have been used is string theory, notably in connection with mirror symmetry in Calabi-Yau manifolds and $D$-branes (see \cite{BBH08} for a good survey). Some of the classical examples of families of Calabi-Yau manifolds for which there is a description of the mirror family are elliptic fibrations \cite{Can91}. Moreover there is a relative Fourier-Mukai transform for most elliptic fibrations \cite{HMP02, HLS08} that can be
understood in terms of duality in string theory \cite{Don98,
DoPan03, DoPan12} or D-brane theory. 
More generally, due to the interpretation of B-type D-branes as objects of the derived category and  to
Kontsevich's  homological mirror symmetry proposal \cite{Kon95}, one expects the
Fourier-Mukai transform (or its relative version) to act on the spectrum of D-branes. The study of D-branes on
Calabi-Yau manifolds inspired in fact the search of new Fourier-
Mukai partners \cite{Or97, Kaw02a, Ue04, HKTY}  among other mathematical problems.

The plan of the paper is as follows. In the second section, we summarize some general results in the theory of Fourier-Mukai partners.  In Section 3, we review some results on the Picard scheme of a projective curve. In Section 4, we compute the Grothendieck and negative $K$-groups of a curve. In Section 5, we recall the definition and main properties of the derived category of singularities. Section 6 collects some geometric properties of Fourier-Mukai partners and in Section 7 we prove Theorem \ref{t1} and Theorem \ref{t2} .
 
This article is partly based on a talk given by the first author at the congress ``VBAC2015: Fourier-Mukai, 34 years on" (University of Warwick, 15Ð19 June 2015).
The authors would like to thank the organizers for the invitation and to the University of Warwick for the hospitality.
\subsubsection*{Conventions}

In this paper, all schemes are assumed to be separated  and quasi-compact over an algebraically closed field $k$ of characteristic zero and unless otherwise stated a point means a closed point.  For any scheme $X$ we denote by $\catD(X)$ the
derived category of complexes of $\cO_X$-modules with
quasi-coherent cohomology sheaves. Analogously $\dpc{X}$, $\dmc{X}$
and $\dbc{X}$ denote the derived categories of complexes
which are respectively bounded below, bounded above and bounded on
both sides, and have quasi-coherent cohomology sheaves. The
subscript $c$ refers to the corresponding subcategories of
complexes with coherent cohomology sheaves. We denote by $\textbf{Pic}_{X}$ the Picard functor and if it is representable then $\Pic(X)$ denotes the representing scheme.
By a curve we will understand a connected curve (possibly reducible or non-reduced) contained in a smooth algebraic surface. We will use then the standard notation for curves lying on smooth algebraic surfaces. If $X$ is a curve over $k$, we will write $X=\sum_{i=1}^{N}m_i\Theta_i$ where $m_i$ are positive integers and $\Theta_i$ are integral curves. For $i=1,\hdots, N$, we will denote $C_i=m_i\Theta_i$ the irreducible components of $X$. The number $m_i$ is called the {\it multiplicity} of $C_i$. Let $m$ be the g.c.d. of the multiplicities $m_i$. If $m>1$, we will say that $X$ is a {\it multiple} curve.

\section{Fourier-Mukai Partners}
Two schemes $X$ and $Y$ are said to be {\it (Fourier-Mukai) partners} if there exists an exact equivalence of triangulated categories between their derived categories $\catD(X)$ and $\catD(Y)$ of quasi-coherent sheaves. For a scheme $X$, denote by $\text{FM}(X)$ the set of isomorphism classes of (Fourier-Mukai) partners of $X$, that is,
$$\text{FM}(X):=\{ Y \text{ schemes  such that } \catD(X)\simeq \catD(Y)\}/\simeq\, .$$ 
Let us obtain other descriptions of  the set $\text{FM}(X)$.

An object $\mathcal{E}$ in $\catD(X)$ is said to be {\it perfect} if it is locally isomorphic to a bounded complex of locally free sheaves of finite type. Denote by $\Perf(X)$ the category of perfect objects on $X$. Obviously, $\Perf(X) \subseteq \cdbc{X}$ and, thanks to Serre's theorem, they are equal if and only if the scheme $X$ is regular.

Perfect complexes on a scheme $X$ are described in purely categorical terms as follows.
An object $a$ in a triangulated category $\mathcal{T}$ is said to be {\it compact} when it commutes with direct sums, that is, if there is an isomorphism $\Hom(a,\oplus_i b_i)\simeq \oplus_i\Hom (a, b_i)$ for each family of objects $\{b_i\}$ in $\mathcal{T}$. Neeman proved  in \cite {Nee96}  that for any scheme $X$ perfect complexes on $X$ are precisely compact objects in $\catD(X)$, that is, $\Perf(X)=\catD(X)^c$. 

Thanks to this categorical characterization of perfect complexes and the existence of dg enhancements, one can prove the following derived Morita theorem (in the sense of Rickard). 
\begin{thm}[\cite{CanStella16},  Proposition 7.4 ] Let $X$ and $Y$ be two schemes with enough locally free sheaves. If $X$ is projective, the following conditions  are equivalent:
\begin{enumerate}
\item There is an exact equivalence $\catD(X)\simeq \catD(Y)$.
\item There is an exact  equivalence $\cdbc{X}\simeq \cdbc{Y}$.
\item  There is an exact equivalence $\Perf(X)\simeq \Perf(Y)$.
\end{enumerate}
\end{thm}\qed

If $X$ and $Y$ are schemes for any object  $\cplx K\in \catD(X\times Y)$, we have an exact functor  $\Phi^{\cplx K}_{X\to Y}\colon \catD(X)\to \catD(Y)$ defined as
$$\Phi_{X\to Y}^{\cplx K}(\cplx{E})=\bR \pi_{Y_\ast}( \pi_X^\ast\cplx{E}\lotimes\cplx{K})\, ,$$
where $\pi_X\colon X\times Y\to X$ and $\pi_Y\colon X\times Y\to Y$ are the two projections. 
The complex $\cplx K$ is said to be the kernel of $\Phi^{\cplx K}_{X\to Y}$. 
Remember that an exact functor $F\colon \catD(X)\to \catD(Y)$ is an {\it integral functor} if there is an object $\cplx K\in \catD(X\times Y)$ and an isomorphism of exact functors $F\simeq \Phi^{\cplx K}_{X\to Y}$. When $\Phi^{\cplx K}_{X\to Y}$ is an equivalence it is called a {\it Fourier-Mukai}  functor.

Due to the famous representability theorem by Orlov \cite{Or97}, if $X$ and $Y$ are smooth projective schemes over $k$, then any exact fully faithful functor $F\colon \cdbc{X}\to\cdbc{Y}$ is an integral functor. A generalization of this result to smooth stacks is contained in \cite{Kaw02}. For a long time it was believed that any exact functor between the bounded derived categories of two smooth projective schemes had to be an integral functor. Nevertheless, a recent counterexample \cite{RiVdB16} by Rizzardo and Van den Bergh shows that this is not true. However, for (not necessarily smooth) projective schemes, we have strong results concerning the representability problem due to Lunts and Orlov. In  \cite{OL10}, we find the following important

\begin{thm}[\cite{OL10}, Theorem 9.9]  Let $X$ be a  projective scheme over $k$ such that the maximal torsion subsheaf  $T_0(\mathcal{O}_X)\subset \mathcal{O}_X$ of dimension 0 is trivial.
Then the triangulated categories $\text{Perf}(X)$ and $\cdbc{X}$ have strongly unique enhancements.
\end{thm}\qed

As a corollary they get the following generalization of Orlov's representabilty theorem in the projective context.

\begin{cor}[Representability] \label{representability} Let $X$ be a projective scheme such that $T_0(\mathcal{O}_X)=0$ and  let $Y$ be any noetherian scheme. Let $F\colon \cdbc{X}\to \cdbc{Y}$ be an exact  fully faithful functor with right adjoint. Then there is an object $\cplx K\in\cdbc{X\times Y}$ and an isomorphism of exact functors $F\simeq \Phi_{X\to Y}^{\cplx K}|_{\cdbc{X}}$.
\end{cor}\qed

 As a consequence of the results contained in this section, 
 the set $\text{FM}(X)$  can be described as follows:
\begin{equation}\label{partners}\begin{aligned}&\text{FM}(X):=\{ Y \text{ such that } \catD(X)\simeq \catD(Y)\}/\simeq\\
&=\{ Y \text{  such that } \Perf(X)\simeq \Perf(Y)\}/\simeq\\
&=\{ Y \text{   such that } \cdbc{X}\simeq \cdbc{Y}\}/\simeq \\
&=\{ Y \text{   such that there is }  \Phi_{X\to Y}^{\cplx K}\colon \cdbc{X} \simeq \cdbc{Y} \}/\simeq
\end{aligned}
\end{equation}

\noindent where the last two equalities are true at least if $X$ is a projective scheme such that $T_0(\mathcal{O}_X)=0$, that is, a projective scheme without embedded points.

For non projective schemes it is not clear yet if every equivalence between the bounded derived categories of coherent sheaves is a Fourier-Mukai functor. See \cite{CanStella11} for a good survey on the subject. 

Replacing the smoothness condition by a Gorenstein condition, C. Sancho de Salas and F. Sancho de Salas generalize the classical Bondal-Orlov reconstruction theorem. They prove the following 

\begin{thm}[Theorem 1.15 in \cite{SS10}] Let $X$ be a connected equidimensional Gorenstein projective scheme over $k$ with ample canonical or antiample canonical sheaf. If $\Perf(X)$ (resp. $\cdbc{X})$ is equivalent as a graded category to $ \Perf(Y)$ (resp. $\cdbc{Y})$ for some other proper scheme $Y$, then $X$ is isomorphic to $Y$.
\end{thm}\qed

The same result is proved by Ballard in \cite{Ballard11}, but in this case the proof uses the triangulated structure either of $\cdbc{X}$ or of $\text{Perf}(X)$.
 
\section{Picard schemes}
For a scheme $X$, let $\Pic(X)=H^1(X,\mathcal{O}_X^*)$ be the Picard group of X, that is, the group of isomorphism classes of line bundles on $X$.  Fixing a base scheme $S$ and an $S$-scheme $X$, the {\it relative Picard functor} $\textbf{Pic}_{X/S}$ is the sheaf associated to the functor that associates to every $S$-scheme $T$, the set $\Pic(X\times_S T)$ for the fppf topology.
If the base scheme $S$ is a field, Grothendieck proved that the Picard functor is representable in the projective case. Later Murre and Oort obtained the representability in the proper case. See \cite{Neron1990} as a reference for the following results:
\begin{thm}Let $X$ be a proper scheme over a field $k$. Then, the Picard functor $\textbf{Pic}_{X/k}$ is representable by a scheme $\Pic(X)$ which is locally of finite type over $k$. 
\end{thm}\qed

Let us recall the structure of the Picard scheme $\Pic(X)$ when $X$ is a proper curve over a field $k$ . Denote by $X_{\text{red}}$  the reduced curve, that is, the largest reduced subscheme of $X$.  By functoriality, we get a canonical map 
$$\Pic(X)\to \Pic(X_{\text{red}})$$ whose kernel and cokernel are described by the following 
\begin{prop}\label{Pic} Let $X$ be a proper curve over a field $k$. Then, the canonical map 
$$\Pic(X)\to \Pic(X_{\text{red}})$$ is an epimorphism of sheaves for the \'etale topology. Its kernel is a smooth and connected unipotent group $U_X$  which is a successive extension of additive groups $\mathbb{G}_a$.
\end{prop} \qed

It remains to give the structure of the Picard scheme of a reduced curve. For the curves we are interested in it is obtained from the following result that we can find in \cite{EGAIV2}, Prop. 21.8.5

\begin{prop}\label{p:grothendieck} Let $X$ and $X'$ be
projective and reduced curves over $k$. Let $\phi\colon X'\to
X$ be a finite and birational \footnote{By a birational morphism
$X'\to X$ of reducible curves we mean a morphism which is an
isomorphism outside a discrete set of points of $X$} morphism. Let
$U$ be the open subset of $X$ such that $\phi\colon
\phi^{-1}(U)\to U$ is an isomorphism and let $S=X-U$. Let us
denote $ \mathcal{O}'_X=\phi_*( \mathcal{O}_{X'})$. Then, there is
an exact sequence
$$0\to (\prod_{s\in S}\mathcal{O}'^{*}_{X,s}/\mathcal{O}_{X,s}^*)/\im H^0(X',\mathcal{O}^*_{X'})\to
\Pic(X)\xrightarrow{\phi^*}\Pic(X')\to 0\, .$$  If the canonical
morphism $H^0(X,\mathcal{O}_X)\to H^0(X',\mathcal{O}_{X'})$
is bijective, then the kernel of $\phi^*$ is isomorphic to
$\prod_{s\in S}\mathcal{O}'^{*}_{X,s}/\mathcal{O}_{X,s}^*$. 
\end{prop}\qed

\begin{cor}\label{c:grupoaditivo} Let $X$ and $X'$ be two
projective reduced and connected curves over $k$. Let
$\phi\colon X'\to X$ be a birational morphism which is an
isomorphism outside $P\in X$. If $\phi^{-1}(P)$ consists of just one point $Q\in
X'$ and $\mathfrak{m}_{X',Q}^2\subset \mathfrak{m}_{X,P}$, then the
sequence
$$0\to\mathbb{G}_a\to \Pic(X)\xrightarrow{\phi^*}\Pic(X')\to 0$$
is exact. 
\end{cor}\qed

\begin{cor}\label{c:odaseshadri}
Let $X=\cup_{i\in I}C_i$ be a projective reduced and connected
curve over $k$. Suppose that the intersection points $\{
P_j\}_{j\in J}$ of its irreducible components are ordinary double
points. Let $X'=\sqcup_{i\in I}C_i$ be the partial normalization
of $X$ at the nodes $\{ P_j\}_{j\in J}$. Then, there is an exact
sequence $$0\to (\mathbb{G}_m)^s\to \Pic(X)\to \Pic(X')\to 0$$ where
$s=|J|-|I|+1$ is the first Betti number of the dual graph of $X$.
\end{cor}\qed

\section{Grothendieck groups and negative $K$-Groups}
In this section, we collect some results concerning the Grothendieck and negative $K$-groups of a scheme $X$. Let us start by recalling some notions.

An {\it exact category} is a pair $(\mathcal{C},\mathcal{E})$ where $\mathcal{C}$ is an additive category with a full embedding $\mathcal{C}\subset \mathcal{A}$ in an abelian category $\mathcal{A}$  and $\mathcal{E}$ is a family of sequences in $\mathcal{C}$  of the form 
\begin{equation}\label{se} 0\to a\to b\to c\to 0
\end{equation} such that  
\begin{enumerate}\item $\mathcal{E}$ is a class of all sequences $\eqref{se}$ in $\mathcal{C}$ which are exact in $\mathcal{A}$.
\item $\mathcal{C}$ is closed under extensions in $\mathcal{A}$ in the sense that if $\eqref{se}$ is an exact sequence in $\mathcal{A}$ and $a,c\in \mathcal{C}$ then $b$  is isomorphic to an object in $\mathcal{C}$.
\end{enumerate}

The exact sequences in $\mathcal{E}$ are called {\it short exact sequences}. We will abuse notation and just say that $\mathcal{C}$ is an exact category when the class $\mathcal{E}$ is clear. 

{\it An exact functor} $F\colon \mathcal{C}\to \mathcal{D}$ between exact categories is an additive functor carrying short exact sequences in $\mathcal{C}$ to short exact sequences in $\mathcal{D}$. 

Let $\mathcal{C}$ be a small exact category. The {\it Grothendieck group} $K_0(\mathcal{C})$ of $\mathcal{C}$ is the abelian group freely generated by the classes $[a]$ of objects $a\in \mathcal{C}$ modulo the relation $[b]=[a]+[c]$ for every short exact sequence  $0\to a\to b\to c \to 0$ in $\mathcal{C}$. 

An exact functor $F\colon \mathcal{A}\to \mathcal{B}$ between exact categories induces a homomorphism of abelian groups $K_0(\mathcal{A})\to K_0(\mathcal{B})$ and equivalent exact categories have isomorphic Grothendieck groups. 

For instance, any abelian category has a natural structure of exact category and if $X$ is a scheme and $\VB(X)$ denotes the category of vector bundles on $X$, then $\VB(X)$ is also an exact category. We will denote  by $G_0(X):=K_0(\text{Coh}(X))$ the Grothendieck group associated to the category of coherent sheaves on $X$ and $K^0(X):=K_0(\VB(X))$ the Grothendieck group associated to the exact category $\text{VB}(X)$ of vector bundles on $X$. 

Let $\mathcal{T}$ be a triangulated category. The {\it Grothendieck group}  $K_0(\mathcal{T})$ of $\mathcal{T}$ is the abelian group freely generated by the classes $[a]$ of objects $a\in \mathcal{T}$ modulo the relation $[b]=[a]+[c]$ for every distinguished triangle $ a\to b\to c \to a[1]$ in $\mathcal{T}$.
Again exact functors (resp. equivalences) between triangulated categories induce group homomorphisms (resp. isomorphisms) between the corresponding Grothendieck groups.

Furthermore, in the particular case that $\mathcal{T}=\catD^b(\mathcal{C})$ is the derived category of an exact category $\mathcal{C}$, there is an isomorphism $K_0(\mathcal{T})\iso K_0(\mathcal{C})$ given by sending $[a]\in K_0(\mathcal{T})$ to $\sum_i(-1)^i[a^i]\in K_0(\mathcal{C}) $ if $a$ is the complex $\hdots \to a^{i-1}\to a^i\to a^{i+1}\to\hdots$. Thus, we have an  isomorphism $G_0(X)\simeq K_0(\cdbc{X})$.

If $X$ is a smooth scheme, every coherent $\mathcal{O}_X$-module has a finite resolution by vector bundles, so that  the Cartan morphism  $K^0(X)\iso G_0(X)$ is a group isomorphism, but this is not the case for singular schemes. However, in many interesting cases, the $K$-theory of vector bundles is equivalent to the $K$-theory  of perfect complexes. Thomason and Trobaugh proved that this is the case for instance for quasi-projective schemes.
\begin{thm}[Corollary 3.4 in \cite{ThomTrob90}] \label{vb} Let $X$ be a quasi-compact and separated scheme with an ample family of line bundles. Then the inclusion of bounded complexes of vector bundles  into $\Perf(X)$ induces an equivalence of triangulated categories $\catD^b(\VB(X))\simeq D(\Perf(X))$. In particular, one has a group isomorphism $$K^0(X)\simeq K_0(\Perf(X))\, .$$
\end{thm}
\qed

Notice that if $X$ is a
scheme and $U \subset X$ an open subscheme, then the restriction of vector
bundles from $X$ to $U$ induces a map of Grothendieck groups $K^0(X)\to K^0(U)$
which in general is not surjective. For this reason,  negative
$K$-groups for schemes are needed. Taking into account the last result, Thomason and Trobaugh  defined {\it the negative $K$-groups} of a scheme $X$ as $K^i(X):=K^i(\Perf(X))$  where $K^i(\Perf(X))$ are defined (following Bass) by the exact sequences 
$$\begin{aligned}
&0\to K^i(\Perf(X))\to K^i(\Perf(X[t]))\oplus K^i(\Perf(X[t^{-1}]))\to\\
&  K^i(\Perf(X[t,t^{-1}]))\to K^ {i-1}(\Perf(X))\to 0\, .
\end{aligned}$$ 
Here  $X[t]$ and $X[t, t^{-1}]$ denote the product of $X$ with $\Spec(\mathbb{Z}[t])$ and $\Spec(\mathbb{Z}[t, t^{-1}])$. 
Negative $K$-groups vanish for smooth schemes, but this is not  true in general for singular schemes.

Let us compute $K^*(X)$ when $X$ is any projective curve. 
If $X$ is a reduced curve, Leslie Roberts proves that $K^{-1}(X)$ can be calculated using its {\it  bipartite graph}   $\Gamma_X$  which is  defined as follows. Let $S_X$ be  the set of singular points of $X$ and $\pi\colon \tilde{X} \to X$ the normalization. For $i=1, \hdots N$, denote by $C_i$  the irreducible components of $X$ and $\tilde{C_i}$ their normalization, so that $\tilde{X}=\sqcup_{i=1}^N \tilde{C_i}$. Then, $\Gamma_X$ has one vertex for each point in $S_X$ and one vertex for each irreducible component of $\tilde{X}$, that is $v_{\Gamma_X}=\{ S_X, \tilde{C_i}, i=1\hdots N\}$. The set of edges of $\Gamma_X$ is obtained as follows: for each point $Q\in \pi^{-1}(S_X)$, we draw an edge connecting the irreducible component $\tilde{C_i}$ of $\tilde{X}$ in which $Q$ lies to the singular point  $\pi(Q)\in S_X$. Let us denote $\lambda(X)$ the number of loops of $\Gamma_X$. Then, one has the following:

\begin{prop}\label{Kcurvas} Let $X$ be a  projective (not necessarily reduced) curve over an algebraically closed field $k$. Then, the following is true
\begin{enumerate} \item The group $K^0 (X)$ is generated by classes of line bundles and there is an isomorphism 
$K^0 (X)\simeq H^0(X,\mathbb{Z})\oplus \Pic(X)$.
\item $K^{-1}(X)\simeq K^{-1}(X_{\text{red}})\simeq \mathbb{Z}^{\lambda}$ where $\lambda=\lambda(X_{\text{red}})$.
\item $K^i(X)=0$ for $i\leq -2$.
\item $X$ is $K^{-1}$-regular, that is, $K^i(X)\simeq K^i(X[t_1,\hdots, t_r])$ for all $r>0$ and all $i\leq -1$.
\end{enumerate}  \end{prop}
\begin{proof}  Since $X$ is a curve for any vector bundle $E$ on $X$ there is a filtration $$ 0=E_s\subset E_{s-1}\subset \hdots \subset E_1\subset E_0=E$$ such that the quotients $E_r/E_{r+1}$ are line bundles. Then the rank and the determinant of a vector bundle define the group isomorphism in (1). Let us prove now the remanning statements. If $X$ is reduced, see \cite{Roberts76}. Otherwise denote by  $\mathcal{I}$  the ideal sheaf defining the closed subscheme $X_{\text{red}}\hookrightarrow X$. We can assume that $\mathcal{I}^2=0$. Otherwise, consider the following filtration $$X_0=X_\text{red}\subset X_1\subset \hdots \subset X_{n-2}\subset X$$ where $X_j$ is the subsheme of $X$ defined by $\mathcal{I}^j$. When $\mathcal{I}^2=0$, one has the following exact sequence of sheaves on $X$
\begin{equation}\label{SucRed}1\to 1+\mathcal{I}\to \mathcal{O}_X^*\to \mathcal{O}_{X_{\text{red}}}^*\to 1\, .
\end{equation}
Since $H^0(X,\mathbb{Z})=H^0(X_{\text{red}},\mathbb{Z})$ and  $K^0(X[t])\simeq H^0(X[t],\mathbb{Z})\oplus \Pic(X[t])$, one concludes by the long cohomology sequence associated to \eqref{SucRed}.
 \end{proof}
 
For a connected and projective curve $X$ whose irreducible
components are isomorphic to $m_i\mathbb{P}^1$ for some positive integer numbers $m_i$, we can also compute the Grothendiek group $G_0(X)$ of coherent sheaves.

\begin{prop} \label{p:grupoK} Let $X$ be a connected and projective curve. If every irreducible
component $C_i$ of $X$ is isomorphic to $m_i\mathbb{P}^1$ where $m_i\geq 1$ is an integer and $\mathbb{P}^1$ is the projective
line, then there is an isomorphism $G_0(X)\simeq
\mathbb{Z}^{N+1}$ where $N$ is the number of irreducible components
of $X$.

\end{prop}
\begin{proof} 
The curve $X$ and the corresponding reduced curve $X_{\text{red}}$ have the same subvarieties, then  there is a canonical isomorphism $A_j(X)\simeq A_j(X_{\text{red}})$  for any $j\in \mathbb{Z}$ \cite[Example 1.3.1]{Ful}. Since  $X_{\text{red}}$ is connected and its irreducible components are
isomorphic to $\mathbb{P}^1$, any two points of $X_{\text{red}}$ are
rationally equivalent. Then $A_0(X)=\bbZ[x]$, where $[x]$ is the
class of a point of  $X_{\text{red}}$. On the other hand it is well known
\cite[Example 1.3.2]{Ful} that the $n$-th Chow group of an
$n$-dimensional scheme is the free abelian group on its
$n$-dimensional irreducible components, therefore
$A_1(X)=\bbZ[\Theta_1]\oplus\dots \oplus\bbZ[\Theta_N]$ where $\Theta_i$ are the irreducible components of $X_{\text{red}}$.

The natural closed immersion  $i\colon X_{\text{red}}\hookrightarrow X$ and the normalization $\pi\colon \widetilde{X_{\text{red}}} \to  X_{\text{red}}$ morphism are both Chow envelopes and, by  \cite[Lemma 18.3]{Ful}, so is the composition  $i\circ \pi\colon  \widetilde {X_{\text{red}}}\to X$.
Thus $\pi_*\colon G_0(\widetilde {X_{\text{red}}}) \to G_0(X)$
is surjective \cite[Lemma~18.3]{Ful}. Since $\widetilde {X_{\text{red}}}=\widetilde
\Theta_1\coprod\dots\coprod\widetilde \Theta_N$ one has
$$G_0(\widetilde {X_{\text{red}}})\simeq G_0(\widetilde
\Theta_1)\oplus\dots\oplus G_0(\widetilde \Theta_N).$$ Taking into
account that $\widetilde \Theta_i\simeq\bbP^1$ we get
$G_0(\widetilde {X_{\text{red}}})=\bigoplus_{i=1}^N\bbZ[\cO_{\widetilde
\Theta_i}]\oplus\bbZ[\cO_{\tilde x_i}]$, where $\tilde x_i\in\widetilde
\Theta_i$. Given any two points $\tilde x_i$, $\tilde y_i$ in $\widetilde
\Theta_i\simeq\bbP^1$ we know that $[\cO_{\tilde x_i}]=[\cO_{\tilde
y_i}]$. Therefore, since $X$ is connected, the surjectivity of
$\pi_*$ implies that $G_0(X)$ is generated by
$$[\cO_{\Theta_1}],\dots, [\cO_ {\Theta_N}], [\cO_{x}].$$
By the Riemann-Roch theorem for algebraic schemes \cite[Theorem
18.3]{Ful} there is a homomorphism $\tau_X\colon G_0(X)\to
A_\ast(X)\otimes_\bbZ\bbQ$ with the following properties:

\begin{enumerate}
\item  For any $k$-dimensional closed subvariety $Y$ of $X$, one has that 
$\tau_X([\cO_Y])=[Y]+ \mathrm{terms\ of\ dimension} < k$, 
\item $(\tau_X)_\bbQ\colon
G_0(X)\otimes_\bbZ\bbQ\xrightarrow{\sim}
A_\ast(X)\otimes_\bbZ\bbQ$ is an isomorphism.
\end{enumerate} 
By the first
property, 
$\{\tau_X([\cO_{\Theta_1}]),\dots,
\tau_X([\cO_{\Theta_N}]),\tau_X([\cO_{x}])\}$ is a basis of
$A_\ast(X)\otimes_\bbZ\bbQ$. Taking into account that
$\{[\cO_{\Theta_1}],\dots,[\cO_{\Theta_N}], [\cO_{x}]\}$ is a system of
generators of $G_0(X)$, the second property of $\tau_X$
implies that it is also a basis.
\end{proof}

Notice that this calculation agrees with the following fact proved using the  {\it D\'evissage} technique (a result that, of course, is not true at the level of derived categories).
\begin{prop}{(D\'evissage)} Let $X$ be any noetherian scheme, the closed immersion $i\colon X_{\text{red}}\hookrightarrow X$ induces a homotopy equivalence of $K$-theories 
$$i_\ast \colon G_0 (X_{\text{red}}) \simeq G_0(X)\, .$$
\end{prop}
\begin{proof}
See, for instance, Theorem 6.3 Chapter II in \cite{Kbook}. 
\end{proof}

\section{The derived category of singularities} 
In this section, we recall the definition and main properties of the derived category of singularities.
This theory developed by Orlov is related to the theory of stable categories of Cohen-Macaulay modules introduced years before by Buchweitz and to
the theory of matrix factorizations given by Eisenbud.
 
We say that a scheme $X$ satisfies the (ELF) condition, if it is separated, noetherian of finite Krull dimension and with enough locally free sheaves. This is satisfied for instance, for any quasi-projective scheme. For a scheme $X$ satisfying the (ELF) condition, Orlov introduced in \cite{Or03} a new invariant, called {\it the category of singularities} of $X$, which is defined as the Verdier quotient triangulated category $$D_\text{sg}(X):=\cdbc{X}/\text{Perf}(X)\, .$$ This is a triangulated category  that reflects the properties of the singularities of $X$. Let us see some properties of it. Denote by $q\colon \cdbc{X}\to D_\text{sg}(X)$ the natural localization functor.

If $\mathcal{D}$ is a triangulated category and $\mathcal{S}$ is a collection of objects in $\mathcal{D}$, we denote $\langle S\rangle_\mathcal{D}$ the smallest thick subcategory of $\mathcal{D}$ containing $\mathcal{S}$. Unifying two results of Orlov, Chen proves in \cite{Chen10} the following result.

\begin{thm}\label{th:chen} Let $X$ be a (ELF) scheme over $k$ and $j\colon U\hookrightarrow X$ an open immersion. Denote by $Z=X$ \textbackslash  $U$ the complement of $U$ and by $Coh_Z(X)\subseteq Coh(X)$ the subcategory of coherent sheaves on $X$ supported in $Z$. Then the induced functor $\bar{j^\ast}\colon  D_\text{sg}(X)\to  D_\text{sg}(U)$ induces a triangle equivalence $$ D_\text{sg}(X)/\langle q(Coh_ZX)\rangle_{D_\text{sg}(X)}\simeq  D_\text{sg}(U)\, .$$\qed
\end{thm}
\
He easily deduces the following two corollaries.

\begin{cor} If we denote $S_X$ the singular locus of $X$, then
\begin{enumerate} \item  $\bar{j^\ast}\colon  D_\text{sg}(X)\to  D_\text{sg}(U)$  is an equivalence if and only if  $S_X\subseteq U$.
\item $D_\text{sg}(X)=\langle q(Coh_ZX)\rangle_{D_\text{sg}(X)}$ if and only if  $S_X\subseteq Z$.
\end{enumerate} \qed
\end{cor}

Let us denote by $k(x)$ the skyscraper sheaf of a closed point $x\in X$, since every coherent sheaf supported in $\{x_1,\hdots,x_n\}$ belongs to $\langle k(x_1)\oplus\hdots\oplus k(x_n)\rangle_{D_\text{sg}(X)}$, one gets the following

\begin{cor}\label{isolated} A scheme $X$ satisfying the (ELF) condition has only isolated singularities $\{x_1,\hdots,x_n\}$ if and only if 
 $$D_\text{sg}(X)\simeq\langle k(x_1)\oplus \hdots\oplus k(x_n)\rangle_{D_\text{sg}(X)}\, .$$ \qed
 \end{cor} 

Another important feature of the derived category of singularities is that  it is not usually idempotent complete. So we will need to consider its idempotent completion constructed as follows.

Remember that an additive category $\mathcal{A}$ is said to be {\it idempotent} if any idempotent morphism $e\colon a\to a$, $e^2=e$, arises from a splitting of $a$, $$a=\Img(e)\oplus\ker(e)\, .$$
 
 If $\mathcal{A}$ is an additive category, the {\it idempotent completion} of $\mathcal{A}$ is the category $\overline{\mathcal{A}}$ defined as follows. Objects of $\overline{\mathcal{A}}$ are pairs $(a, e)$ where $a$ is an object of $\mathcal{A}$ and $e\colon a\to a$ is an idempotent morphism. A morphism in $\overline{\mathcal{A}}$ from $(a, e)$ to $(b,f)$ is a morphism $\alpha\colon  a\to b$ such that $\alpha\circ e=f\circ \alpha=\alpha$.
 
 There is a canonical fully faithful functor $i\colon \mathcal{A}\to \overline{\mathcal{A}}$ defined by sending $a\to (a,1_a)$.
 
 The following theorem can be found in \cite{BalSch01}.
 
 \begin{thm}\label{t:completion} Let $\mathcal{D}$ be a triangulated category. Then its idempotent completion $\overline{\mathcal{D}}$ admits a unique structure of triangulated category such that the canonical functor $i\colon \mathcal{D}\to \overline{\mathcal{D}}$ becomes an exact functor. If $\overline{\mathcal{D}}$ is endowed with this structure, then for each idempotent complete triangulated category $\mathcal{C}$ the functor $i$ induces an equivalence
$$\Hom(\overline{\mathcal{D}},\mathcal{C})\simeq \Hom(\mathcal{D},\mathcal{C})$$
where $\Hom$ denotes the category of exact functors.
 \end{thm}\qed

To finish this section, let us remember the structure of the category of singularities for a Gorenstein scheme with isolated singularities. Let $(A,\mathfrak{m})$ be a local commutative  ring of Krull dimension $d$. Recall that  a finitely generated $A$-module $M$ is called a {\it maximal Cohen Macaulay} module if $\text{depth}(M)=d$ or equivalently $\text{Ext}^i(A/\mathfrak{m}, M)=0$ for all $i<d$. The category of all maximal Cohen-Macaulay modules over $A$ is denoted by $\CM(A)$. We will also denote $\sCM(A)$ the {\it stable category of Cohen-Macaulay modules} over $A$ which is defined as follows:
\begin{itemize}\item Ob$(\sCM(A))$=Ob($\CM(A)$).
\item $\Hom_{\sCM(A)}(M,N)=\Hom_{\CM(A)}(M,N)/\mathcal{P}(M,N)$ where $\mathcal{P}(M,N)$ is the submodule of $\Hom_{\CM(A)}(M,N)$ generated by those morphisms which factor through a free $A$-module.
\end{itemize}

When $(A,\mathfrak{m})$ is Gorenstein,  Buchweitz proves that the stable category of Cohen-Macaulay modules 
 $\underline{CM}(A)$ is a triangulated category and there is an exact equivalence  $\underline{CM}(A)\simeq D_{sg}(A)$ with the derived category of singularities. 
\

 The following result describes the category of singularities for Gorenstein schemes with isolated singularities. It is a well-known result to experts (see \cite{BK11}, \cite{Or09} or \cite{IyWe11}).

 \begin{thm}\label{t:equivalencia} Suppose that $X$ is a Gorenstein scheme over $k$ satisfying the (ELF) condition and only with isolated singularities $\{x_1,\hdots, x_n\}$. Then,
 there is a fully faithful functor $D_\text{sg}(X)\hookrightarrow \oplus_{i=1}^n \sCM(\mathcal{O}_{x_i})$ which is essentially surjective up to direct summands. Taking the idempotent completion gives an equivalence of categories $$\overline{D_\text{sg}(X)}\simeq \oplus_{i=1}^n\sCM(\hat{\mathcal{O}} _{x_i})$$ where $\hat{\mathcal{O}} _{X,x_i}$ is the completion of the local ring $\mathcal{O}_{X,x_i}$ at the point  $x_i$.
 \end{thm} \qed

\section{Geometric consequences of an equivalence}\label{propiedades}

The existence of a Fourier-Mukai functor between the derived categories of two schemes has important geometric  consequences not only in the smooth but also in the singular case. In this section we give those that we will need for our proposals.

Let $X$ and $Y$ be two projective Fourier-Mukai partners. Then the following is true:
\begin{enumerate}
\item If $X$ is connected, $Y$ is connected.
\item  If $X$ is a smooth scheme, then $Y$ is also smooth  (see for instance \cite{Bri98t}).
\item  If $X$ is an equidimensional Cohen-Macaulay  (resp. Gorenstein) scheme and $Y$ is reduced, then $Y$ is also equidimensional and it is  Cohen-Macaulay (resp. Gorenstein) as well (Theorem 4.4 in \cite{HLS08}).
\item $X$ and $Y$ have the same dimension (Theorem 4.4 in \cite{HLS08}).
\item If $X$ is Gorenstein,  $\omega_X$ is trivial if and only if $\omega_Y$ is trivial (Proposition 1.30 in \cite{HLS07}).
\item If  $X$ and $Y$ are Gorenstein, there is an  isomorphism  $$\oplus_{i\geq 0}H^0(X, \omega_X^i)\simeq \oplus_{i\geq 0}H^0(Y, \omega_Y^i)$$
between  the graded canonical algebras (Proposition 2.1 \cite{LM14}). In particular, if $X$ and $Y$ are two curves, then they have the same arithmetic genus.
\end{enumerate}
In the same vein, we have the following results.

\begin{cor}\label{isoK} Let $X$ and $Y$ be two projective Fourier-Mukai partners. For any $i\leq 0$, there exist group isomorphisms
$$G_0(X)\simeq G_0(Y)\quad\text{ and }\quad K^i(X)\simeq K^i(Y)$$ 
between their Grothendieck groups and their negative $K$-groups.
\end{cor}

\begin{proof} By the characterization of the Fourier-Mukai partners given in \eqref{partners}, one has exact equivalences $\cdbc{X} \simeq \cdbc{Y}$ and $\Perf(X)\simeq \Perf(Y)$. Since both schemes are quasi-projective, one concludes using Theorem \ref{vb}.
\end{proof}

\begin{cor} \label{c:Pic} If  $X$ and $Y$ are two connected projective curves over an algebraically closed field $k$ that are Fourier-Mukai partners, then there is a group isomorphism
$$\Pic(X)\simeq \Pic(Y)$$ between their Picard schemes.
\end{cor}
\begin{proof} This follows from Corollary \ref{isoK} and Proposition \ref{Kcurvas}.
\end{proof}
\begin{lem} \label{Dsin} Let $X$  and $Y$ be two (ELF) schemes over $k$. If $F\colon \catD(X)\to \catD(Y)$ is an exact equivalence, then there is an exact equivalence $$\bar{F}\colon D_\text{sg}(X)\to D_\text{sg}(Y)$$ between their derived categories of singularities.
\end{lem}

\begin{proof} If $\mathcal{C}'$ and ${\mathcal D}'$ are full triangulated subcategories in triangulated categories $\mathcal{C}$ and $\mathcal{D}$ and $F\colon \mathcal{C}\to \mathcal{D}$ and $G\colon \mathcal{D}\to \mathcal{C}$ is an adjoint pair of exact functors satisfying $F({\mathcal C}')\subset {\mathcal D}'$ and $G({\mathcal D}')\subset {\mathcal C}'$, it is easy to prove that the induced functors $$\overline{F}\colon {\mathcal C}/{\mathcal C}'\to {\mathcal D}/{\mathcal D}',\quad \overline{G}\colon {\mathcal D}/{\mathcal D}' \to {\mathcal C}/{\mathcal C}'$$ are also an adjoint pair.  Neeman's characterization of perfect complexes on $X$  as compact objects of the derived category $\catD(X)$ of quasi-coherent sheaves, proves that for any exact essentially surjective functor $F\colon \catD(X)\to \catD(Y)$ one has $F(\Perf(X))\subseteq \Perf(Y)$. Thus, being $F$ an exact equivalence, the above discussion proves that there is an induced exact equivalence $\bar{F}\colon D_\text{sg}(X)\to D_\text{sg}(Y)$.
\end{proof}

 \begin{cor} \label{c:isolated} Let $X$ be a quasi-projective scheme with isolated singularities. If  $Y\in FM(X)$  is a Fourier-Mukai partner of $X$, then $Y$ has also isolated singularities and both schemes have the same number of singular points.
 \end{cor}

 \begin{proof} This follows from the above lemma using the description of $ D_\text{sg}(X)$ given by Corollary \ref{isolated}.  
 \end{proof}

\section{Kodaira curves and their Fourier-Mukai partners}
 Let $p\colon S\to B$ be a smooth elliptic surface, that is, a flat and projective morphism from a smooth surface $S$ to a curve $C$  whose generic fiber is a smooth elliptic curve. The classification of all possible singular fibers of $p$ dates back to Kodaira \cite{kod}, who also named them. The following list ${\bf L} $ contains all possible fibers divided in three subclasses ${\bf (L_j)}$. From now on we will refer to the fibers in ${\bf L}$ as {\it Kodaira curves}. 
  
\begin{itemize}\item[${\bf (L_1)}$] Those fibers that are {\bf reduced curves}, that is, 

\

\begin{enumerate}
\item[$(I_0)$] a smooth elliptic curve.
\item[ ($I_1$)]  a rational curve with one node.
\item[$(I_N)$] a cycle of $N$ rational smooth curves with $N\geq 2$.
\item[$(II)$] a rational curve with one cusp.
\item[$(III)$] two rational smooth curves forming a tacnode curve.
\item[$(IV)$] three concurrent rational smooth curves in the plane.
\end{enumerate}

\

\item[${\bf (L_2)}$] Those fibers that are {\bf non-multiple and non-reduced curves},  that is, 

\

\begin{enumerate}\item[$(I_N^*)$]  $X=\Theta_1+\Theta_2+\Theta_3+\Theta_4+2\Theta_5+2\Theta_6+\hdots+2\Theta_{N+5}$  for $N\geq 0$ and such that 
$(\Theta_1\cdot \Theta_5)=(\Theta_2\cdot \Theta_5)=(\Theta_3\cdot \Theta_{N+5})=(\Theta_4\cdot \Theta_{N+5})=(\Theta_5\cdot \Theta_6)=\hdots=(\Theta_{N+4}\cdot \Theta_{N+5})=1$.

\

\item[$(II^*)$] $X=\Theta_1+2\Theta_2+3\Theta_3+4\Theta_4+5\Theta_5+6\Theta_6+4\Theta_7+3\Theta_8+2\Theta_9$ where 
$(\Theta_1\cdot \Theta_2)=(\Theta_2\cdot \Theta_3)=(\Theta_3\cdot \Theta_4)=(\Theta_4\cdot \Theta_5)=(\Theta_5\cdot \Theta_6)=(\Theta_6\cdot \Theta_8)=(\Theta_6\cdot \Theta_7)=(\Theta_7\cdot \Theta_9)=1$.

\

\item[$(III^*)$] $X=\Theta_1+2\Theta_2+3\Theta_3+4\Theta_4+3\Theta_5+2\Theta_6+2\Theta_7+\Theta_8$ where $(\Theta_1\cdot \Theta_2)=(\Theta_2\cdot \Theta_3)=(\Theta_3\cdot \Theta_4)=(\Theta_4\cdot \Theta_6)=(\Theta_4\cdot \Theta_5)=(\Theta_5\cdot \Theta_7)=(\Theta_7\cdot \Theta_8)=1$.

\

\item[$(IV^*)$] $X=\Theta_1+2\Theta_2+3\Theta_3+2\Theta_4+2\Theta_5+\Theta_6+\Theta_7$ where $(\Theta_1\cdot \Theta_2)=(\Theta_2\cdot \Theta_3)=(\Theta_3\cdot \Theta_4)=(\Theta_3\cdot \Theta_5)=(\Theta_4\cdot \Theta_6)=(\Theta_5\cdot \Theta_7)=1$.

\end{enumerate}

\

\item[${\bf (L_3)}$] Those fibers that are {\bf multiple curves},  that is, 

\

\begin{enumerate} \item[$(_mI_0)$] $X=mC$ where $C$ is a smooth elliptic curve and $m\geq 2$.

\

\item [$(_mI_1)$] $X=mD$ where $D$ is a rational curve with one node and $m\geq 2$.

\

\item [$(_mI_N)$] $X=m\Theta_1+\hdots +m\Theta_N$ where $m\geq 2$ and  $(\Theta_1\cdot \Theta_2)=(\Theta_2\cdot \Theta_3)=\hdots =(\Theta_{N}\cdot \Theta_1)=1$ and $m\geq 2$.
\end{enumerate}
\end{itemize}

\qed

All the irreducible components $\Theta_i$  of the reducible fibers are smooth rational curves with self intersection -2. 

Recall that the {\it arithmetic genus} of a projective curve $X$ is defined by $g_a(X):=1-\chi(\mathcal{O}_X)$ where $\chi$ is the Euler-Poincar\'e characteristic. Any Kodaira curve $X$ is a Gorenstein projective curve of arithmetic genus one and has trivial dualising sheaf.

Reduced Kodaira curves were characterized by Catenese in \cite{Cata82} as follows.

\begin{prop}[Proposition 1.18 in  \cite{Cata82}]\label{Cat} Let  $X$ be a projective planar reduced connected curve over an algebraically closed field $k$. Then, $X$ is Gorenstein of arithmetic genus one and has trivial dualising sheaf if and only if $X$ is isomorphic to a Kodaira curve in the subclass ${\bf (L_1)}$.
\end{prop}\qed

As far as we know there is no similar classification for non-reduced Gorenstein curves of genus one and trivial dualising sheaf in the literature.

Let us compute now the Grothendieck groups and the negative $K$-groups of all Kodaira fibers.

\begin{prop} \label{kodairaK} Let $X$ be a Kodaira curve. Then the Grothendieck group $G_0(X)$ of coherent sheaves on $X$ is isomorphic to $
\mathbb{Z}^{N_X+1}$ where $N_X$ is the number of irreducible components of $X$.
\end{prop}
\begin{proof} If $X_{\text{red}}$ is an integral curve (that is, a smooth elliptic curve,  a rational curve with one node or a rational curve with one cusp), the rank and the degree of a coherent sheaf on $X_{\text{red}}$ give the isomorphism $G_0(X)\simeq \mathbb{Z}^2$. Otherwise, since all irreducible components of $X$ are isomorphic to $m_i\mathbb{P}^1$ for some integer  $m_i>0$ and $\mathbb{P}^1$ the projective line, the result follows from Proposition \ref{p:grupoK}.
\end{proof}

\begin{prop}\label{kodairaKnegativo} Let  $X$ be a Kodaira curve. \begin{enumerate}
\item If $X_\text{red}$ is of type $I_N$ for some $N\geq 1$,  then  $K^{-1}(X)\simeq\mathbb{Z}$ and $K^i(X)=0$ for any $i< -1$.
\item Otherwise,  $K^i(X)=0$ for any $i\leq -1$.
\end{enumerate}
\end{prop}
\begin{proof} This follows from Proposition \ref{Kcurvas} and the fact that $\Gamma_{X_{\text{red}}}$ has one loop if $X_\text{red}$ is of type $I_N$ and it has no loops for any other Kodaira curve. 
\end{proof}

Our next aim is to know more about Fourier-Mukai partners of Kodaira curves. In this sense, we have the following

\begin{thm} \label{t1}  Let $X$ and $Y$ be two Kodaira curves such that there is an exact equivalence $\Phi\colon \catD(X)\simeq \catD(Y)$ between their derived categories of quasi-coherent sheaves. Then,
\begin{enumerate}
\item  $X\in {\bf (L_j)}$ if and only if $Y\in {\bf (L_j)}$ for ${\bf j}=1,2,3$.
\item If  ${\bf j}=1$, then $X$ is isomorphic to $Y$. In this case, $\FM(X)=\{X\}$.
\end{enumerate}
\end{thm}
 
\begin{proof} Notice first that since $X$ is a projective curve without embedded points, $T_0(\mathcal{O}_X)=0$ and consequently all descriptions of the set $\FM(X)$  given in \eqref{partners} are valid and $Y\in \FM(X)$.

Thus, the two Fourier-Mukai partners $X$ and $Y$ share all the geometric properties stated in Section 6. In particular, 
the isomorphism $G_0 (X)\simeq G_0 (Y)$ between the Grothendieck groups of coherent sheaves implies,  by Proposition \ref{kodairaK}, that $X$ and $Y$ have the same number $N_X=N_Y$ of irreducible components. Furthermore, by Corollary \ref{c:Pic}, we have an isomorphism $\phi\colon \Pic(X)\simeq \Pic(Y)$ between their Picard schemes.

Suppose first that $X$ belongs to ${\bf (L_1)}$. If $X$ is a smooth curve, then $Y$ has to be also smooth and the isomorphism between their Picard schemes induces an isomorphism $\Pic^0(X)\simeq \Pic^0(Y)$ between their Jacobians. Since the Jacobian of a smooth elliptic curve is isomorphic to itself,  we have $X\simeq Y$ and one concludes the proof of the theorem in the smooth case. If $X$ is singular, since $X$ has only isolated singularities, by Corollary \ref{c:isolated}  the same is true for $Y$ which means that $Y$ is also reduced. Then $Y$ belongs to ${\bf (L_1)}$ because of  Proposition \ref{Cat}. Moreover, Corollary \ref{c:isolated} implies that  $X$ and $Y$ have the same number of singular points.
Having the same number of irreducible components and the same number of singular points, to conclude that $\FM(X)=\{X\}$ the only thing to check is that a rational curve with one node (type $I_1$) and a rational curve with one cusp (type $III$) cannot be Fourier-Mukai partners. Let $C$ be a curve of type $I_1$ and let $D$ be a curve of type $III$.
By Corollaries \ref{c:grupoaditivo} and \ref{c:odaseshadri}, we have  the exact sequences
$$0\to \mathbb{G}_m\to \Pic{C}\to \Pic{\mathbb{P}^1}\to 0$$
$$0\to \mathbb{G}_a\to \Pic{D}\to \Pic{\mathbb{P}^1}\to 0 $$
where $ \mathbb{G}_m$ (resp. $\mathbb{G}_a$) is the multiplicative (resp. additive) group. Then,
 $\Pic(C)$ is not isomorphic to  $\Pic(D)$ and we 
finish the proof in this case.

Suppose now that $X$ belongs to ${\bf (L_2)}$.  The Picard scheme of $X$ is given by an exact sequence 
\begin{equation}\label{PicL2}
0\to \mathbb{G}_a\to \Pic(X)\to \mathbb{Z}^{N_X}\to 0\, .
\end{equation}
Indeed,  in this case $X_{\text{red}}$ is always a tree-like curve (see \cite{LM05}). Then, by Corollary \ref{c:odaseshadri}, one has $$\Pic(X_{\text{red}})\simeq \prod_{i=1}^{N_X}\Pic (\mathbb{P}^1)\simeq \mathbb{Z}^{N_X}$$

According to Proposition \ref{Pic}, the  kernel of the epimorphism $\Pic(X)\to \Pic(X_\text{red})$ is a uniponent group $U_X$ which is a successive extension of additive groups of dimension equal to $h^1(X, \mathcal{O}_X)-h^1(X_\text{red},\mathcal{O}_{X_{\text{red}}})$. Since $X$ is connected and $\omega_X\simeq \mathcal{O}_X$, one gets $U_X\simeq \mathbb{G}_a$.

On the other hand,  if $Y$ belongs to ${\bf (L_3)}$, one has an isomorphism $\Pic(Y)\simeq \Pic(Y_\text{red})$ because in this case the kernel $U_Y$ has dimension $h^1(Y, \mathcal{O}_Y)-h^1(Y_\text{red},\mathcal{O}_{Y_{\text{red}}})=0$. Then, by Corollary  \ref{c:odaseshadri}, the Picard scheme of $Y$ is given by an exact sequence 
\begin{equation}\label{PicL3}
0\to \mathbb{G}_m\to \Pic(Y)\to \mathbb{Z}^{N_Y}\to 0\,. 
\end{equation}
By \eqref{PicL2} and \eqref{PicL3}, there are no isomorphisms between $\Pic(X)$ and  $\Pic(Y)$. Thus,  $X$ and $Y$ cannot be Fourier-Mukai partners finishing the proof.
\end{proof}

In the reduced case it is the derived category of singularities $D_\text{sg}(X)$ what allows to distinguish the Fourier-Mukai partners. If $X$ has only isolated singularities, this category is well described by Theorem \ref{t:equivalencia}. Very little is known about this category in the case of schemes with non-isolated singularities. A first property is given by the  following 

\begin{thm} \label{t2} Let $X$ be a Kodaira curve in the subclass  ${\bf (L_2)}$. Then the derived category of singularities $D_\text{sg}(X)$ of $X$ is idempotent complete.
\end{thm}
\begin{proof} The triangulated categories $\cdbc{X}$ and $\Perf(X)$ are both idempotent complete and we have a full triangle embedding $\Perf(X)\subset \cdbc{X}$. By Proposition \ref{kodairaKnegativo}, we know that $K^{-1}(X)=0$ for any Kodaira curve $X$ in the subclass  ${\bf (L_2)}$. According to  Remark 1 in \cite{Schilchting06} this vanishing result implies that the Verdier quotient $D_\text{sg}(X)= \cdbc{X}/\Perf(X)$ is idempotent complete.
\end{proof}

 \bibliographystyle{siam}

\def\cprime{$'$}

  \end{document}